\documentclass[12pt, a4paper, reqno]{amsart}

\usepackage{amsmath,amssymb,amsthm,amsfonts,url,color,mathrsfs, csquotes, enumerate}

\usepackage[lmargin=2.5 cm,rmargin=2.5 cm,tmargin=3.5cm,bmargin=2.5cm,paper=a4paper]{geometry}

\usepackage{color}
\definecolor{bleu_sombre}{rgb}{0,0,0.6}
\definecolor{rouge_sombre}{rgb}{0.8,0,0}
\definecolor{vert_sombre}{rgb}{0,0.6,0}
\usepackage[plainpages=false,colorlinks,linkcolor=bleu_sombre,citecolor=rouge_sombre,urlcolor=vert_sombre,breaklinks]{hyperref}

\theoremstyle{plain}
\newtheorem{theorem}{{Theorem}}[section] 
\newtheorem*{theorem*}{{Theorem}}
\newtheorem{lemma}[theorem]{Lemma}
\newtheorem*{lemma*}{Lemma}

\theoremstyle{remark}
\newtheorem{remark}[theorem]{Remark}
\newtheorem{notation}[theorem]{Notation}
\newtheorem{assumption}[theorem]{Assumption}

\makeatletter

\@addtoreset{equation}{section}  
\makeatother

\renewcommand{\leq}{\leqslant}	\renewcommand{\geq}{\geqslant}

\newcommand{\R}{\mathbb{R}}	
\newcommand{\C}{\mathbb{C}}
\newcommand{\N}{\mathbb{N}}	

\newcommand{\dx}{\mathrm{d}}

\renewcommand{\Re}{\mathrm{Re}\,}
\renewcommand{\Im}{\mathrm{Im}\,}

\usepackage[english]{babel}

\begin{document}

\title[Magnetic WKB constructions]{WKB constructions\\ in bidimensional magnetic wells}

\author{Y. Bonthonneau}
\email{yannick.bonthonneau@univ-rennes1.fr}
\author{N. Raymond}
\email{nicolas.raymond@univ-rennes1.fr}
\address[Y. Bonthonneau \& N.Raymond]{Univ Rennes, CNRS, IRMAR - UMR 6625, F-35000 Rennes, France}

\date{\today}

\begin{abstract}
This article establishes, in an analytic framework and in two dimensions, the first WKB constructions describing the eigenfunctions of the pure magnetic Laplacian with low energy when the magnetic field has a unique minimum that is positive and non-degenerate.
\end{abstract}

\maketitle

\section{Spectral theory of the magnetic Laplacian}

\subsection{Motivation and context}

\subsubsection{Definition of the magnetic Laplacian}
Let $\Omega$ be a bounded open set of $\R^2$ with $(0,0)\in\Omega$. Let us consider a closed $2$-form, analytic in a neighborhood of $\Omega$, denoted by $\sigma$ and called magnetic $2$-form. We write 
\[\sigma=B\dx x_{1}\wedge \dx x_{2}\,,\] 
and we call $B$ the magnetic field. We first pick a gauge. Let us consider an analytic and real function $\varphi$ such that, in a neighborhood of $\Omega$,
\[\Delta\varphi=B\,,\quad \mbox{ and }\quad\varphi(x_{1},x_{2})=\frac{B(0,0)}{4}(x_{1}^2+x^2_{2})+\mathscr{O}(\|x\|^3)\,.\]
Then $\mathbf{A}=\nabla\varphi^\perp=(-\partial_{x_{2}}\varphi,\partial_{x_{1}}\varphi)$ is an analytic vector potential associated with $B$, that is
\[B=\partial_{x_{1}}A_{2}-\partial_{x_{2}}A_{1}\,.\]
In other words, with $\pi= A_{1}\dx x+A_{2}\dx y$, we have $\sigma=\dx\pi$. With this choice, we have 
\[\nabla\cdot\mathbf{A}=0\,.\]
The magnetic Laplacian $\mathscr{L}_{h}$ under consideration in this article is the self-adjoint realization on $L^2(\Omega)$ with Dirichlet boundary condition of the following differential operator
\[(-ih\nabla-\mathbf{A})^2=(hD_{x_{1}}-A_{1})^2+(hD_{x_{2}}-A_{2})^2\,,\quad D=-i\partial\,.\]

\subsubsection{Semiclassical magnetic spectrum}
The spectral analysis of the magnetic Laplacian $\mathscr{L}_{h}$ has undergone recent important developments. For an introduction to this vast subject, the reader might want to consult the book by the second author \cite{Ray17}. There are many reasons to consider the spectral theory of $\mathscr{L}_{h}$. Initially, it was motivated by the study of the Ginzburg-Landau theory and the estimates of its critical fields which are directly related to the asymptotic behavior of the first eigenvalue $\lambda_{0}(h)$ (see the book \cite{FH10}). But, it also acquired a life of its own. Among the wide literature developed in the last ten years, the works by Helffer and Kordyukov \cite{HK11, HK14} are the most closely related to the subject of the present article (and they are strong improvements of \cite[Theorem 7.2]{HM96}, see the review paper \cite{HK09}). In particular, when the magnetic field admits a unique, non-degenerate  and positive, minimum at $(0,0)$, they prove the following asymptotic expansions for the eigenvalues at the bottom of the spectrum (see \cite[Theorem 1.2]{HK11}):
\begin{equation}\label{eq.HK}
\forall \ell\in\mathbb{N}\,,\quad\lambda_{\ell}(h)=b_{0}h+\left(2\ell\frac{\sqrt{\det H }}{b_{0}}+\frac{(\mathrm{Tr}\, H^{\frac{1}{2}})^2}{2b_{0}}\right)h^2+o(h^2)\,,
\end{equation}
where $H=\frac{1}{2}\mathrm{Hess}_{(0,0)} B$. This result is generalized to Riemanian manifolds and the uniformity of the asymptotics with respect to $\ell$ is improved thanks to a pseudo-differential dimensional reduction in \cite{HK14}. Whereas the proofs of these results involve various (hypo-)elliptic estimates in the semiclassical limit, no connection between the semiclassical estimates and the classical dynamics is made. In \cite{RVN15} the authors link the eigenvalues expansions \eqref{eq.HK} with the Hamiltonian dynamics. The argument relies on the use of Birkhoff normal forms and corresponding quantization via Fourier Integral Operators. Note that the three-dimensional case has also recently been investigated thanks to this point of view in \cite{HKRVN16}.

\subsubsection{Aim of the article}
The aim of the article is to solve the following open question (mentioned for example in the lecture \cite[Section 6.1]{H09}), in the analytic case:
\begin{center}
\enquote{Are the eigenfunctions associated with the eigenvalues \eqref{eq.HK} in a WKB form?}
\end{center}
At first, it can be surprising that such a basic question finds no answer in the existing literature. The only known results of this nature were obtained recently in a multi-scale framework (see \cite{BHR16}), but the case of the \textit{purely magnetic wells} and when no scaling consideration allows to reduce the dimension, was still left open. For the sake of comparison, the reader may consult \cite[Section 6]{H09} or \cite[Chapter 3]{DS99} about the WKB constructions in the purely electric case.

The motivation to answer our magnetic question, under the generic assumption of Helffer and Kordyukov, comes from the analysis of tunneling effect when the magnetic field has two symmetric minima. Until now and contrary to the purely electric situation (see for instance \cite{HJ84}), there is no result giving the accurate estimate of $\lambda_{1}(h)-\lambda_{0}(h)$, called tunneling effect, and there is not even an explicit conjecture of what it could be (as a comparison, the WKB constructions of \cite{BHR16} were turned into an explicit conjecture \cite[Conjecture 1.4]{BHR16b} which is now numerically checked). We only expect it to be exponentially small when $h$ goes to zero. A necessary step to get such a result is the approximation of the eigenfunctions, in an appropriate exponentially weighted space, by an explicit (WKB) Ansatz. Our computation is the first step in this direction.

\subsubsection{Heuristics}
Nevertheless, it would not be quite accurate to say that there is no conjecture for the WKB constructions. Let us sketch the result of \cite{RVN15}.

There exist a Fourier Integral Operator $U_{h}$, quantizing a canonical transformation, and a smooth function $f_{h}$ such that, locally in space near $0$ and microlocally near the characteristic manifold of $\mathscr{L}_{h}$,
\[
U_{h}^* \mathscr{L}_{h}U_{h}=\mathrm{Op}^{\mathsf{w}}_{h} f_{h}(\mathcal{H},z_{2})+\mathscr{O}(h^\infty)\,.
\]
where $\mathcal{H}=h^2D_{x_{1}}^2+x^2_{1}$. Moreover, $f_{h}(Z,z_{2})=Z \hat{B}(z_{2})+\mathscr{O}(h^2)+\mathscr{O}(Z^2)$, where $\hat B$ is the magnetic field \enquote{seen} on the characteristic manifold. Thus, if we are interested in the low lying eigenvalues (which are essentially in the form $b_{0}h+\mu_{1}h^2$), we can look for a $L^2$-normalized WKB Ansatz expressed in normal coordinates as
\[
\Psi_{h}(x_{1},x_{2})=g_{h}(x_{1})\psi_{h}(x_{2})\,,
\]
where $g_{h}$ is the first normalized eigenfunction of $\mathcal{H}$. We find the effective eigenvalue equation
\[
\mathrm{Op}^\mathsf{w}_{h}(\hat B-b_{0})\psi_{h}=\mu_{1}h\psi_{h}+\mathscr{O}(h^2)\,,
\]
in which we insert the Ansatz $\psi_{h}=e^{-S/h}a$. We get
\begin{equation}\label{eq.heuristics}
\hat B(x_{2},-iS'(x_{2}))=b_{0}\,.
\end{equation}
Therefore, in canonical coordinates, the phase should be the sum of the phase of $g_{h}$ and of the phase determined by \eqref{eq.heuristics}. It is then not very difficult to write the transport equation in the variable $x_{2}$ to find $a$ and guess that the amplitude of the WKB construction is the product of the amplitude of $g_{h}$ and of $a(x_{2})$.

While FIO's preserve WKB states, the use of Birkhoff normal forms in the construction of $U_h$ implies that the remainders are not as good as one can get by direct constructions. Additionally, $U_h$ is not explicit, so the link between the coefficients of the states and the original magnetic field is quite implicit. However we will see that the point of view developed in \cite{RVN15} gives a reasonable insight of the rigorous WKB constructions.

\subsection{WKB construction in a magnetic well}

\begin{assumption}
$B_{|\overline{\Omega}}$ has a non-degenerate local and positive minimum at $(0,0)$. Moreover, we can write
\begin{equation}\label{eq.DLB}
B(x_{1},x_{2})=b_{0}+\alpha x^2_{1}+\gamma x_{2}^2+\mathscr{O}(\|x\|^3)\,,\quad \mbox{ with }0<\alpha\leq\gamma\,.
\end{equation}
\end{assumption}
Of course, \eqref{eq.DLB} is always satisfied up to an appropriate choice of coordinates. The result of this paper is

\begin{theorem*}
Let $\ell\in\mathbb{N}$. There exist 
\begin{enumerate}[\rm i.]
\item a neighborhood $\mathcal{V}\subset\Omega$ of $(0,0)$, 
\item an analytic function $S$ on $\mathcal{V}$ satisfying
\[\Re S(x)=\frac{b_{0}}{2}\left[\frac{\sqrt{\alpha}}{\sqrt{\alpha}+\sqrt{\gamma}}x^2_{1}+\frac{\sqrt{\gamma}}{\sqrt{\alpha}+\sqrt{\gamma}}x^2_{2}\right]+\mathscr{O}(\|x\|^3)\,,\]
\item a sequence of analytic functions $(a_{j})_{j\in\mathbb{N}}$ on $\mathcal{V}$, 
\item a sequence of real numbers $(\mu_{j})_{j\in\mathbb{N}}$ satisfying
\[\mu_{0}=b_{0}\,,\quad \mu_{1}=2\ell\frac{\sqrt{\alpha\gamma}}{b_{0}}+\frac{(\sqrt{\alpha}+\sqrt{\gamma})^2}{2b_{0}}\,,\]
\end{enumerate}
such that, for all $J\in\mathbb{N}$, and uniformly in $\mathcal{V}$,
\[e^{S/h}\left((-ih\nabla-\mathbf{A})^2-h\sum_{j\geq 0}^{J}\mu_{j} h^j\right)\left(e^{-S/h}\sum_{j\geq 0}^{J} a_{j} h^j\right)=\mathscr{O}(h^{J+2})\,.\]
\end{theorem*}

\begin{remark}
Considering a convenient cutoff function supported near the origin and using the local exponential decay of $e^{-S/h}$, our Ansatz can be used as a quasimode for $\mathscr{L}_{h}$. Therefore, if we assume that the minimum of $B_{|\overline{\Omega}}$ is unique, thanks to the spectral theorem and \eqref{eq.HK}, we get the expansion of the first eigenvalues \eqref{eq.HK} at any order. Due to their asymptotic simplicity, this also proves that our WKB expansion are approximations, in the $L^2$-sense, of the corresponding eigenfunctions. \end{remark}

\subsection{Organization and methods}
Section \ref{sec.prelim} is devoted to convenient lemmas which will allow to lighten the presentation of the proof of the  theorem when determining the phase $S$. In Section \ref{sec.proof}, we prove the theorem. We will see that the eikonal equation will not be enough to determine the phase of the Ansatz contrary to the purely electric case. The holomorphic part of the phase will only be determined when solving the first complexified transport equation on $a_{0}$. The transport equation on $a_{1}$ will then be necessary to find the full expression of $a_{0}$. The two complex transport equations on $a_{j}$ and $a_{j+1}$ are the keys to construct the Ansatz and they reflect the classical dynamics in a magnetic field. Their characteristic curves are related to the cyclotron and center guide motions. These dynamical properties appear, in our presentation, in terms of division arguments in the ring of analytic functions of two variables.

\section{Analytic preliminaries about the magnetic phase}\label{sec.prelim}

\begin{lemma}\label{lem.phi}
There exists an analytic and real-valued function $\varphi$, in a neighborhood of $\Omega$, such that 
\[\Delta\varphi=B\,,\quad \varphi(x_{1},x_{2})=\frac{B(0,0)}{4}(x_{1}^2+x^2_{2})+\mathscr{O}(\|x\|^3)\,.\]
\end{lemma}
\begin{proof}
If we write
\[B(x_{1},x_{2})=\sum_{(\alpha,\beta)\in\N^2} a_{\alpha,\beta}x^\alpha_{1}x_{2}^\beta\,,\]
we choose
\[\varphi(x_{1}, x_{2})=\frac{1}{2}\left(\sum_{(\alpha,\beta)\in\N^2} \frac{a_{\alpha,\beta}}{(\alpha+1)(\alpha+2)}x_{1}^{\alpha+2}x_{2}^\beta+\sum_{(\alpha,\beta)\in\N^2} \frac{a_{\alpha,\beta}}{(\beta+1)(\beta+2)}x^\alpha_{1}x_{2}^{\beta+2}\right)\,.\]
It satisfies the required property --- and has the same radius of convergence as $B$.
\end{proof}

\begin{notation}
If $a : \R^2\to\C$ is an analytic function near $(0,0)\in\R^2$, one denotes by $\tilde a$ the function defined near $(0,0)\in\C^2$ by
\[\tilde a (z,w)=a\left(\frac{z+w}{2},\frac{z-w}{2i}\right)\,.\]
We have $\tilde a(z,\overline{z})=a(\Re z,\Im z)$.
\end{notation}

\begin{lemma}\label{lem.eik-eff}
There exists a holomorphic function $w$ defined in a neighborhood of $0$ satisfying 
\begin{equation}\label{eq.eik-eff}
\tilde B(z,w(z))=b_{0}\,.
\end{equation} 
and such that
\[w(0)=0\,,\quad w'(0)=\frac{\sqrt{\gamma}-\sqrt{\alpha}}{\sqrt{\gamma}+\sqrt{\alpha}}\,.\]
\end{lemma}
\begin{proof}
Let us use the Taylor formula:
\[
 B(x_{1},x_{2})-b_{0}= \alpha(x_{1},x_{2})x_{1}^2+2\beta(x_{1},x_{2})x_{1}x_{2}+\gamma (x_{1},x_{2})x_{2}^2\,,
\]
where $\alpha$, $\beta$ and $\gamma$ are analytic functions such that $\alpha(0,0)=\alpha$, $\beta(0,0)=0$ and $\gamma(0,0)=\gamma$.
We get
\[
 B(x_{1},x_{2})-b_{0}= \alpha\left[\left(x_{1}+\frac{\beta}{\alpha}x_{2}\right)^2+\left(\frac{\alpha\gamma-\beta^2}{\alpha^2}\right)x_{2}^2\right]\,.
\]
Thus, we consider the equations
\[\sqrt{\alpha}\left(x_{1}+\frac{\beta}{\alpha}x_{2}\right)\pm i\sqrt{\frac{\alpha\gamma-\beta^2}{\alpha}}x_{2}=0\,.\]
We replace $x_{1}$ by $\frac{z+w}{2}$ and $x_{2}$ by $\frac{z-w}{2i}$. Equation \eqref{eq.eik-eff} becomes
\[
\left[\tilde\alpha+i\tilde\beta\mp \sqrt{\tilde\alpha\tilde\gamma-\tilde\beta^2}\right]w+\left[\tilde\alpha-i\tilde\beta\pm\sqrt{\tilde\alpha\tilde\gamma-\tilde\beta^2}\right]z=0\,.
\]
Let us choose the $+$ in the first bracket so that, at $(0,0)$ it is equal to $\alpha+\sqrt{\alpha\gamma}>0$. By using the analytic implicit function theorem, we find a holomorphic solution $w$. By a straightforward computation, one gets
\[[\alpha+\sqrt{\alpha\gamma}]w'(0)+[\alpha-\sqrt{\alpha\gamma}]=0\,,\]
and the conclusion follows.
\end{proof}

\begin{lemma}\label{lem.inv-loc}
Consider a holomorphic function $F$ defined in a neighborhood of $0$ with $F(0)=0$. Then, there exist two neighborhoods of $0$, $\mathcal{V}_{1}$ and $\mathcal{V}_{2}$ such that, for all $z\in\mathcal{V}_{1}$, there exists a unique $w(z)\in\mathcal{V}_{2}$ such that
\[\partial_{z}\tilde\varphi(z,w)=F(z)\,.\]
Moreover, the function $w$ is holomorphic on $\mathcal{V}_{1}$.
\end{lemma}
\begin{proof}
We recall that $\varphi$ is analytic and that $4\partial_{w}\partial_{z}\tilde\varphi(0,0)=\tilde B(0,0)\neq 0$. The conclusion follows then from the (holomorphic) local inversion theorem.
\end{proof}

\begin{lemma}\label{lem.f}
Consider a function $w$ as in Lemma \ref{lem.eik-eff} and, in a neighborhood of $0$, the holomorphic function defined by
\[f(z)=-2\int_{[0,z]}\partial_{z}\tilde\varphi(\zeta,w(\zeta))\dx \zeta\,.\]
We have
\[f(0)=0\,,\quad f'(0)=0\,,\quad f''(0)=\frac{b_{0}}{2}\frac{\sqrt{\alpha}-\sqrt{\gamma}}{\sqrt{\gamma}+\sqrt{\alpha}}\,.\]
In particular, letting $S=\varphi+f$, we have
\[\Re S(x)=\frac{b_{0}}{2}\left[\frac{\sqrt{\alpha}}{\sqrt{\alpha}+\sqrt{\gamma}}x^2_{1}+\frac{\sqrt{\gamma}}{\sqrt{\alpha}+\sqrt{\gamma}}x^2_{2}\right]+\mathscr{O}(\|x\|^3)\,.\]
\end{lemma}
\begin{proof}
A straightforward computation gives
\[f''(0)=-2\partial^2_{z}\tilde\varphi(0,0)-2\partial_{z}\partial_{w}\tilde\varphi(0,0)w'(0)\,.\]
Noticing that, by our choice of $\varphi$, $\partial^2_{z}\tilde\varphi(0,0)=0$, we get the announced value of $f''(0)$. It remains to write that
\[\Re S(x)=\frac{b_{0}}{4}(x_{1}^2+x_{2}^2)+q\frac{b_{0}}{4}\Re\left((x_{1}+ix_{2})^2\right)+\mathscr{O}(\|x\|^3)\,,\quad q=\frac{\sqrt{\alpha}-\sqrt{\gamma}}{\sqrt{\gamma}+\sqrt{\alpha}}\,,\]
and we get
\[\Re S(x)=\frac{b_{0}}{4}\left((1+q)x_{1}^2+(1-q)x_{2}^2)\right)+\mathscr{O}(\|x\|^3)\,.\]
\end{proof}

\section{Proof of the theorem}\label{sec.proof}
Let us consider an analytic and complex-valued function $S$, defined in a neighborhood of the origin. We consider the conjugated operator acting locally as 
\[\mathscr{L}^S_{h}=e^{S/h}\mathscr{L}_{h}e^{-S/h}=(hD_{x_{1}}-A_{1}+i\partial_{x_{1}}S)^2+(hD_{x_{2}}-A_{2}+i\partial_{x_{2}}S)^2\,.\]
We have
\[\mathscr{L}^S_{h}=(-A_{1}+i\partial_{x_{1}}S)^2+(-A_{2}+i\partial_{x_{2}}S)^2+ih\nabla\cdot\mathbf{A}-h^2\Delta+h\Delta S+2h(\nabla S+i\mathbf{A})\cdot \nabla\,.\]
We seek to determine $S$ so that there exist a family of functions $(a_{j})_{j\in\mathbb{N}}$ defined in a neighborhood of $(0,0)$ and a sequence of real numbers $(\mu_{j})_{j\in\mathbb{N}}$ such that, in the sense of asymptotic series,
\begin{equation}\label{eq.formal-eve}
\mathscr{L}^S_{h}\left(\sum_{j\geq 0} h^ja_{j}\right) \sim h\left(\sum_{j\geq 0}\mu_{j}h^j\right)\left(\sum_{j\geq 0} h^ja_{j}\right)\,.
\end{equation}
From \eqref{eq.formal-eve}, we get an infinite system of partial differential equations.

\subsection{Eikonal equation}
Collecting the terms of order $1$ in \eqref{eq.formal-eve}, we get 
\[(-A_{1}+i\partial_{x_{1}}S)^2+(-A_{2}+i\partial_{x_{2}}S)^2=0\,,\]
and thus
\[(-A_{1}+i\partial_{x_{1}}S+i(-A_{2}+i\partial_{x_{2}}S))(-A_{1}+i\partial_{x_{1}}S-i(-A_{2}+i\partial_{x_{2}}S))=0\,.\]
Let us consider an $S$ such that
\[-A_{1}+i\partial_{x_{1}}S+i(-A_{2}+i\partial_{x_{2}}S)=0\,.\]
It satisfies
\[2\partial_{\overline{z}}S=-iA_{1}+A_{2}\,,\qquad \partial_{\overline{z}}=\frac{1}{2}\left(\partial_{x_{1}}+i\partial_{x_{2}}\right)\,.\]
We have $2\partial_{\overline{z}}\varphi=-iA_{1}+A_{2}$ and thus $S$ is in the form
\[S=\varphi+f(z)\,,\]
where $f$ is a holomorphic function near $(0,0)$. Note that $\Delta=4\partial_{z}\partial_{\overline{z}}$ and thus
\[\Delta S=B-i\nabla\cdot\mathbf{A}=B\,.\]
With this choice of $S$, we have
\[\mathscr{L}^S_{h}=-h^2\Delta+hB+2h(\nabla S+i\mathbf{A})\cdot \nabla\,.\]
We have
\[(\nabla S+i\mathbf{A})\cdot\nabla=(\partial_{1}S+iA_{1})\partial_{1}+(\partial_{2}S+iA_{2})\partial_{2}\]
so that
\[(\nabla S+i\mathbf{A})\cdot\nabla=(\partial_{1}\varphi-i\partial_{2}\varphi+f'(z))\partial_{1}+(\partial_{2}\varphi+i\partial_{1}\varphi+if'(z))\partial_{2}\,.\]
Therefore, we can write
\[\mathscr{L}^S_{h}=-4h^2\partial_{z}\partial_{\overline{z}}+hB+4h(2\partial_{z}\varphi+f'(z))\partial_{\overline{z}}\,,\]
and consider its complexified extension
\[\mathscr{L}^S_{h}=hv(z,w)\partial_{w}+hB-4h^2\partial_{z}\partial_{w}\,,\quad v(z,w)=8\partial_{z}\tilde\varphi(z,w)+4f'(z)\,,\]
acting on analytic functions of $(z,w)\in\C^2$.

\subsection{Study of the transport operator}

The PDE's solved by the family $(a_j)_{j\in\N}$ take the form of a family of transport equations. We will need the following lemma.
\begin{lemma}\label{lemma.transport-equation}
Let $V$ and $F$ be two holomorphic functions defined around $0$. Assume that $V(0)=0$ and $V'(0)\neq 0$, and consider the transport equation
\[
(V(z)\partial_z + F(z))f(z) = g(z)\,.
\]
\begin{enumerate}[\rm i.]
\item The homogeneous equation --- $g=0$ --- has holomorphic solutions around $0$ if and only if there exists $\ell\in\N$ such that $F(0) = -\ell V'(0)$. In this case, the solutions vanish at the order $\ell$ at $0$. 

\item Under the previous condition, there exist complex numbers $(c_k)_{k=0\dots\ell}$ such that the inhomogeneous equation has holomorphic solutions if and only if
\begin{equation}\label{eq.cond-c}
c_{\ell} g(0) + c_{\ell - 1} g'(0) + \dots + c_0 g^{(\ell)}(0) = 0\,.
\end{equation}
The coefficients are determined by the Taylor expansion to order $\ell+1$ of $F$ and $V$, and $c_0 = 1/V'(0)$. When $\ell =0$, provided the condition \eqref{eq.cond-c} is satisfied, the inhomogeneous equation has exactly one solution vanishing at $0$.
\end{enumerate}

\end{lemma}

\begin{proof}
Let us start with the homogeneous case. Consider a non-zero solution $f$. We can always write $f(z) = z^\ell \widehat{f}(z)$, where $\widehat{f}(0)\neq 0$. Then we find
\[
( \ell V'(0)  + F(0)) \widehat{f}(0) = 0\,,
\]
so that $\ell V'(0) + F(0) = 0$. Now, if $\ell V'(0) + F(0) = 0$, we write the equation in the form
\[
\frac{\ell}{z} + \frac{\widehat{f}'}{\widehat{f}} = - \frac{F}{V}\,,
\]
and, since $V'(0)$ does not vanish, we can write
\[
\frac{F}{V} = - \frac{\ell}{z} + G\,,
\]
where $G$ is a holomorphic function. We deduce that there are solutions, and they take the form
\begin{equation}\label{eq.sol-homogeneous}
f(z)= f^{(\ell)}(0)\frac{z^\ell}{\ell !} \exp\left\{- \int_0^z G(z') \dx z' \right\}\,.
\end{equation}
Now, we turn to the inhomogeneous case. We can always write the solutions in the form
\[
f(z) = \widehat{f}(z)\exp\left\{- \int_0^z G(z') \dx z' \right\}\,.
\]
The equation for $\widehat{f}$ is
\[
(z\partial_z - \ell)\widehat{f} = \frac{ z g }{V} \exp\left\{ \int_0^z G(z') \dx z' \right\}\,.
\]
By considering the Taylor expansions at $0$, we deduce that a necessary and sufficient condition to have holomorphic solutions is 
\[
\partial^{(\ell)}_z \left[ \frac{ z g }{V} \exp\left\{ \int_0^z G(z') \dx z' \right\}\right] = 0\,.
\]
This relation is in the form \eqref{eq.cond-c}.

When $\ell = 0$, we can divide the equation by $z$ and obtain a usual non-singular ODE for $f$. There is a unique solution that vanishes at $0$.
\end{proof}

\subsection{First transport equation}
The first transport equation, obtained by gathering the terms of order $h$, is
\begin{equation}\label{eq.transport1}
(\tilde v(z,w)\partial_{w}+\tilde B(z,w)-\mu_{0})\tilde a_{0}=0 \,.
\end{equation}
The fact that this equation needs to have solutions will determine $f$. 

\subsubsection{Choosing $f$}
Let us for now assume that $f$ is given and let $\underline{w}$ be, by Lemma \ref{lem.inv-loc}, the unique (holomorphic and local) solution of
\begin{equation}\label{eq.stationnaire}
8\partial_{z}\tilde\varphi(z,\underline{w}(z))+4f'(z)=0\,.
\end{equation}
By freezing the variable $z$, and after a translation by $-\underline{w}$ in the $w$ variable, we can apply Lemma \ref{lemma.transport-equation}. We deduce that  \eqref{eq.transport1} has solutions if and only if the exists $\ell\in \N$ such that
\[
\tilde{B}(z,\underline{w}(z)) - \mu_0 = - \ell \partial_w \tilde{v}(z,\underline{w}(z))\,.
\]
But, from the definition of $\tilde v$, this means
\[
\mu_0 = (2\ell + 1) \tilde B(z,\underline{w}(z))\,.
\]
Since $\mu_0$ is a constant, we deduce that $\mu_{0}=(2\ell+1)b_{0}$ and 
\begin{equation}\label{eq.eiktB}
\tilde B(z,\underline{w}(z)) = b_{0}\,.
\end{equation}
Locally, there may be more than one solution to \eqref{eq.eiktB}, but we choose $\underline{w}(z) = w(z)$, where $w(z)$ is given by Lemma \ref{lem.phi}. With this choice for $\underline{w}$, we define $f$ as the unique function such that $f(0)=0$ and 
\begin{equation}\label{eq.f}
f'(z)  = - 2 \partial_z \varphi(z,w(z))\,.
\end{equation}

\subsubsection{Solving the transport equation}
We notice that
\[
\frac{ \tilde B(z,w) - b_0}{ 8\partial_z \varphi(z,w) + 4 f'(z)}
\]
defines a holomorphic function near $(0,0)$. Considering Lemma \ref{lemma.transport-equation}, and particularly \eqref{eq.sol-homogeneous}, the solutions of \eqref{eq.transport1} have to take the form
\[
\tilde{a}_0(z,w) = \partial_w^{(\ell)}\tilde{a}_0(z,w(z)) \frac{(w-w(z))^\ell}{\ell !} \exp\left[ - \int_{w(z)}^w \left(\frac{\tilde B-\mu_0}{\tilde{v}} + \frac{\ell}{w'-w(z)}\right) \dx w' \right]\,.
\]
We denote 
\begin{equation}\label{eq.J_m}
J_\ell(z,w): = \exp\left[ - \int_{w(z)}^w \frac{\tilde B-\mu_0}{\tilde{v}}(z,w') + \frac{\ell}{w'-w(z)} \dx w' \right]\,,
\end{equation}
and then, the function
\[
\tilde{a}_0(z,w) = \mathscr{A}_0(z) (w-w(z))^\ell J_\ell(z,w)\,,
\]
solves \eqref{eq.transport1} with $\mu_{0}=(2\ell+1)b_{0}$. The function $\mathscr{A}_{0}$ is a holomorphic function to be determined.

We are chiefly interested in the low-lying eigenvalues, so we will consider the smallest value possible for $\mu_0$, and thus restrict our attention to the case $\ell = 0$. We write $J_0 = J$.

\subsection{Second transport equation}
The equation obtained by gathering the terms in $h^2$ can be written as
\begin{equation}\label{eq.transport2}
(\tilde v(z,w)\partial_{w}+\tilde B(z,w)-\mu_{0})\tilde a_{1}=\left(\mu_{1}+4\partial_{z}\partial_{w}\right)\tilde a_{0}\,.
\end{equation}
This equation will determine $\mathscr{A}_0$ and $\mu_1$. Indeed, applying Lemma \ref{lemma.transport-equation}, this time for the inhomogeneous case, we deduce that this equation has solutions if and only if
\[\left(\mu_{1}+4\partial_{z}\partial_{w}\right)\tilde a_{0}(z,w(z))=0\,.\]
This means that
\begin{equation}\label{eq.eff-transpA0}
4\mathscr{A}_{0}'(z)\partial_{w}J(z,w(z))+\left[\mu_{1}+4\partial_{w}\partial_{z}J(z,w(z))\right]\mathscr{A}_{0}(z)=0\,.
\end{equation}
This is also a transport equation, but in the $z$ variable this time. We want to apply Lemma \ref{lemma.transport-equation}, so we compute the coefficients of the equation, at least at $0$. Observe that
\begin{equation}\label{eq.dwJ}
\partial_{w}J(z,w)=J(z,w)\frac{\tilde B(z,w(z))-\tilde B(z,w)}{8\left[\partial_{z}\tilde\varphi(z,w)-\partial_{z}\tilde\varphi(z,w(z))\right]}\,.
\end{equation}
We get
\[
\partial_{w}J(z,w(z)) = - \frac{\partial_{w}\tilde B(z,w(z))}{2\tilde B(z,w(z))}\,,
\]
so that
\[4\partial_{w}J(z,w(z))\underset{z\to 0}{\sim}-\frac{4\partial_{z}\partial_{w}\tilde B(0,0)+4\partial^2_{w}\tilde B(0,0)w'(0)}{2b_{0}}z\,.\]
We have
\[4\partial_{z}\partial_{w}\tilde B(0,0)=2(\alpha+\gamma)\,,\quad 4\partial^2_{w}\tilde B(0,0)=2(\alpha-\gamma)\,,\]
and thus
\begin{equation}\label{eq.dwJ00}
4\partial_{w}J(z,w(z))\underset{z\to 0}{\sim}-\left[\alpha+\gamma+(\alpha-\gamma)\frac{\sqrt{\gamma}-\sqrt{\alpha}}{\sqrt{\alpha}+\sqrt{\gamma}}\right]\frac{z}{b_{0}}=-2\sqrt{\alpha\gamma}\frac{z}{b_{0}}\,.
\end{equation}
Additionally, we notice that
\[
-\frac{2}{b_{0}}\sqrt{\alpha\gamma}=4\partial_{w}\partial_{z}J(0,0)+4\partial^2_{w}J(0,0)w'(0)\,,
\]
and, thanks to \eqref{eq.dwJ} and the Taylor formula, we get that
\[\partial^2_{w}J(0,0)=\frac{\gamma-\alpha}{8b_{0}}\,.\]
Thus,
\begin{equation}\label{eq.4dwdzJ00}
4\partial_{w}\partial_{z}J(0,0)=-\frac{2}{b_{0}}\sqrt{\alpha\gamma}-\frac{\gamma-\alpha}{2b_{0}}\frac{\sqrt{\gamma}-\sqrt{\alpha}}{\sqrt{\gamma}+\sqrt{\alpha}}=-\frac{(\sqrt{\alpha}+\sqrt{\gamma})^2}{2b_{0}}\,.
\end{equation}
We now apply Lemma \ref{lemma.transport-equation} to Equation \eqref{eq.eff-transpA0}. With \eqref{eq.dwJ00} and \eqref{eq.4dwdzJ00}, we get that there exists $\ell \in \N$ such that
\begin{equation}\label{eq.mu1}
\mu_{1}=2\ell\frac{\sqrt{\alpha\gamma}}{b_{0}}+\frac{(\sqrt{\alpha}+\sqrt{\gamma})^2}{2b_{0}}\,.
\end{equation}
Then, by using \eqref{eq.sol-homogeneous}, we can write $\mathscr{A}_{0}(z)= c z^\ell \widehat{\mathscr{A}}_{0}(z)$, where $\widehat{\mathscr{A}}_0(z)$ is determined with $\widehat{\mathscr{A}}_{0}(0) =1$. The constant $c$ is a normalization constant, we choose $c=1$.

The solutions of Equation \eqref{eq.transport2} take the form
\[\tilde a_{1}(z,w)=\hat a_{1}(z,w)+\mathscr{A}_{1}(z)J(z,w)\,,\]
where $\mathscr{A}_{1}$ remains to be determined and $\hat a_{1}$ is the particular solution that vanishes for $w=w(z)$.

\begin{remark}
If we write the characteristics of \eqref{eq.eff-transpA0}, the obtained dynamics reflects the center guide motion whose approximate Hamiltonian is $\tilde B(z,w(z))$.
\end{remark}

\subsection{Induction}
Let $n\in\mathbb{N}\setminus\{0\}$. We assume that the $(\mu_{j})_{0\leq j\leq n}$ and the $(\tilde a_{j})_{0\leq j\leq n-1}$ are determined and that the $(\tilde a_{j})_{0\leq j\leq n-1}$ are analytic functions. Let us also assume that the $\tilde a_{j}$'s, $j=1\dots n$, are in the form
\[
\tilde a_{j}(z,w)=\hat a_{j}(z,w)+\mathscr{A}_{j}(z)J(z,w)\,,
\]
where $\hat a_{j}$ are determined analytic functions vanishing on $\{w=w(z)\}$, $(\mathscr{A}_j)_{j=1\dots n-1}$ are determined and satisfy $\mathscr{A}_j^{(\ell)}(0) = 0$. Only $\mathscr{A}_{n}$ is still to be determined. Let us now consider the equation satisfied by $\tilde a_{n+1}$:
\begin{equation}\label{eq.an+1}
\left(v\partial_{z}+\tilde B-b_{0}\right)\tilde a_{n+1}=\mu_{n+1}\tilde a_{0}+\mu_{1}\tilde a_{n}+4\partial_{z}\partial_{w}\tilde a_{n}+\sum_{j=2}^{n}\mu_{j}\tilde a_{n+1-j}\,.
\end{equation}
As before, the need to have solutions to this equation will fix the value of $\mu_{n+1}$ and determine $\mathscr{A}_n$. Indeed, by Lemma \ref{lemma.transport-equation}, the existence of solutions to \eqref{eq.an+1}, is equivalent to
\[
\mu_{1}\tilde a_{n}(z,w(z)) + 4\partial_{z}\partial_{w}\tilde a_{n}(z,w(z)) =- \mu_{n+1}\tilde a_{0}(z,w(z)) - \sum_{j=2}^{n}\mu_{j}\tilde a_{n+1-j}(z,w(z))\,.
\]
This can be rewritten as
\begin{multline}\label{eq.effn+1}
4\partial_{w}J(z,w(z))\mathscr{A}'_{n}(z)+\left(\mu_{1}J(z,w(z))+4\partial_{z}\partial_{w}J(z,w(z))\right)\mathscr{A}_{n}(z)\\
=-\mu_{n+1}z^\ell \widehat{\mathscr{A}}_{0}(z) - \underset{:=F(z)}{\underbrace{\sum_{j=2}^n \mu_j \mathscr{A}_{n+1-j}(z)}}\,,
\end{multline}
We are in the inhomogeneous case of Equation \eqref{eq.eff-transpA0}. We already know that Lemma \ref{lemma.transport-equation} applies. The function $F$ is entirely determined already and $F^{(\ell)}(0)=0$. In particular, there are coefficients $c_\ell,\dots,c_0$ depending on the Taylor expansion to order $\ell+1$ of $\partial_{w}J(z,w(z))$ and $\partial_{z}\partial_{w}J(z,w(z))$, with $c_0 \neq 0$ --- we can even compute it to be $-b_0/\sqrt{4\alpha\gamma}$ --- such that there are solutions to \eqref{eq.effn+1} if and only if
\[
 \mu_{n+1} c_0\ell ! + c_1 F^{(\ell - 1)}(0) + \dots + c_\ell F(0) = 0\,.
\]
This determines $\mu_{n+1}$. However, $\mathscr{A}_{n}$ is now determined up to a solution of the homogeneous equation \eqref{eq.eff-transpA0}. That is to say that $\mathscr{A}_n$ takes the form
\[
\mathscr{A}_n^0 + c z^\ell \widehat{\mathscr{A}}_0\,.
\]
where $\mathscr{A}_n^0$ is a particular solution. There is only one such solution with $\mathscr{A}_n^{(\ell)}(0) = 0$, and that is the one we pick. 

Coming back to $\tilde{a}_{n+1}$, with this choice of $\mu_{n+1}$ and $\mathscr{A}_n$, there are solutions to Equation \eqref{eq.an+1} and they can be written as:
\[\tilde a_{n+1}(z,w)=\hat a_{n+1}(z,w)+\mathscr{A}_{n+1}(z)J(z,w)\,,\]
where $\hat a_{n+1}$ is a determined analytic function vanishing on $\{ w = w(z) \}$ and $\mathscr{A}_{n+1}$ is a function to determine. By induction, we can thus build the desired holomorphic functions, and the proof of the theorem is complete.\qed

\begin{remark}
It may seem arbitrary to have imposed that $\mathscr{A}_n^{(\ell)}(0)=0$ when $n\neq 0$. However, consider that the whole quasimode writes out formally as
\[
\tilde{u}_\ell : = e^{-\frac{S}{h}}J(z,w)\left[\mathscr{A}_h(z) + h \hat{a}_h(z,w) \right]\,,
\]
with 
\[
\mathscr{A}_h \sim z^\ell \widehat{\mathscr{A}}_0  + h \mathscr{A}_1 + h^2 \mathscr{A}_2 + \dots
\]
and
\[
\hat{a}_h \sim \hat{a}_1 + h\hat{a}_2 + \dots\,.
\]
The condition we have imposed is equivalent to the normalization condition that if $\mathscr{U}(z)$ is the restriction of $e^{S/h}\tilde{u}_\ell$ to $\{ w = w(z)\}$, $\mathscr{U}^{(\ell)}(0) = \ell !$.
\end{remark}


\end{document}